\documentclass[article,10pt]{elsarticle}

\usepackage{amssymb}
\usepackage{amsmath}
\usepackage{amsfonts}
\usepackage{amsthm}
\usepackage{graphics,graphicx}
\usepackage{float}
\usepackage{mathtools}

\usepackage[english]{babel}
\journal{Arxiv}
\newtheorem{theorem}{Theorem}
\newtheorem{lemma}{Lemma}

\newtheorem{corollary}{Corollary}
\newtheorem{example}{Example}

\newtheorem{problem}{Problem}

\begin{document}

\begin{frontmatter}

\title{Optimal functions with spectral constraints in hypercubes}

\author[01]{Alexandr Valyuzhenich}
\ead{graphkiper@mail.ru}


\address[01]{Moscow Institute of Physics and Technology, Dolgoprudny, Russia}


\begin{abstract}
The $n$-dimensional hypercube has $n+1$ distinct eigenvalues $n-2i$, $0\leq i\leq n$, with corresponding eigenspaces $U_i(n)$.
In 2021 it was proved by the author that if a function with non-empty support belongs to the direct sum $U_i(n)\oplus U_{i+1}(n)\oplus\ldots\oplus U_j(n)$,
where $0\leq i\leq j\leq n$, then it has at least $\max(2^i,2^{n-j})$ non-zeros.
In this work we give a characterization of functions achieving this bound.
\end{abstract}

\begin{keyword}
hypercube\sep eigenfunction\sep eigenfunctions of graphs\sep minimum support\sep trade\sep $[t]$-trade
\vspace{\baselineskip}
\MSC[2010] 05C50\sep 05B30
\end{keyword}

\end{frontmatter}

\section{Introduction}\label{Sec:Intro}
There are the following extremal problems for eigenfunctions of graphs.

\begin{problem}\label{Problem:1}
Let $G$ be a graph and let $\lambda$ be an eigenvalue of $G$. Find the minimum cardinality of the support of a $\lambda$-eigenfunction of $G$.
\end{problem}

\begin{problem}\label{Problem:2}
Let $G$ be a graph and let $\lambda$ be an eigenvalue of $G$. Characterize $\lambda$-eigenfunctions of $G$ with the minimum cardinality of the
support.
\end{problem}

During the last years, Problems \ref{Problem:1} and \ref{Problem:2} have been actively studied for various families of distance-regular graphs
\cite{B18,GKSV18,GSY23,K16,KMP16,S18,S19,V17,VV19,V21,VK15,VMV18} and Cayley graphs on the symmetric group \cite{KKSV20}.
In particular, Problem \ref{Problem:1} is completely solved for all eigenvalues of the Hamming graph \cite{K16,VV19,V21} and asymptotically solved for all eigenvalues of the Johnson graph \cite{VMV18}. In more details, Problems \ref{Problem:1} and \ref{Problem:2} are discussed in a recent survey \cite{SV21}.

The {\em Hamming graph} $H(n,q)$ is defined as follows. The vertex set of $H(n,q)$ is $\mathbb{Z}_{q}^n$, and two vertices are adjacent if they differ in exactly one coordinate.
The adjacency matrix of $H(n,q)$ has $n+1$ distinct eigenvalues $n(q-1)-q\cdot i$, where $0\leq i\leq n$.
Let $U_{[i,j]}(n,q)$, where $0\leq i\leq j\leq n$, denote the direct sum of eigenspaces of $H(n,q)$ corresponding to consecutive eigenvalues from $n(q-1)-q\cdot i$ to $n(q-1)-q\cdot j$. The support of a real-valued function $f$ is denoted by $S(f)$.

Let $0\leq i\leq j\leq n$. Denote $$m_{i,j}(n,q)=\min_{f\in U_{[i,j](n,q)},f\not\equiv 0}|S(f)|.$$
A function $f\in U_{[i,j]}(n,q)$ is called {\em optimal} in the space $U_{[i,j]}(n,q)$ if $|S(f)|=m_{i,j}(n,q)$.
In this work we consider the following natural generalizations of Problems \ref{Problem:1} and \ref{Problem:2} for the Hamming graph.
\begin{problem}\label{Problem:3}
Let $n\geq 1$, $q\ge 2$ and $0\leq i\leq j\leq n$. Find $m_{i,j}(n,q)$.
\end{problem}

\begin{problem}\label{Problem:4}
Let $n\geq 1$, $q\ge 2$ and $0\leq i\leq j\leq n$. Characterize functions that are optimal in the space $U_{[i,j]}(n,q)$.
\end{problem}

Problem \ref{Problem:3} is completely solved for all $n\geq 1$ and $q\ge 2$ in \cite{VV19,V21}.
Moreover, Problem \ref{Problem:4} is solved for $q\ge 3$, $i+j\le n$ and $q\ge 5$, $i=j$, $i>\frac{n}{2}$ in \cite{VV19}.
In this work we solve Problem \ref{Problem:4} for $q=2$ and arbitrary $n$.
The main ideas of the proof are the following.
For $i+j\geq n$, we prove that functions that are optimal in the space $U_{[i,j]}(n,2)$ correspond to some $[i-1]$-trades in $H(n,2)$ (for more information on $[t]$-trades see \cite{GKKK20,K18}).
Then we apply a characterization of $[t]$-trades of size $2^{t+1}$ obtained by D. Krotov in \cite{K18}.
Finally, using the bipartiteness of $H(n,2)$, we reduce the case $i+j\leq n$ to the case $i+j\geq n$.

The paper is organized as follows. In Section \ref{Sec:Definitions}, we introduce basic definitions.
In Section \ref{Sec:Prelim}, we give preliminary results.
In Section \ref{Sec:Constructions}, we present constructions of functions that are optimal in the space $U_{[i,j]}(n,2)$.
In Section \ref{Sec:Main}, we characterize functions that are optimal in the space $U_{[i,j]}(n,2)$.
In Section \ref{Sec:Spectr}, we discuss the properties of the spectrum of optimal functions.

\section{Basic definitions}\label{Sec:Definitions}
The eigenvalues of a graph are the eigenvalues of its adjacency matrix.
Let $G$ be a graph with vertex set $V$ and let $\lambda$ be an eigenvalue of $G$. The set of neighbors of a vertex $x$ is denoted by $N(x)$.
A function $f:V\longrightarrow{\mathbb{R}}$ is called a {\em $\lambda$-eigenfunction} of $G$ if $f\not\equiv 0$ and the equality
\begin{equation}\label{Eq:Eigenfunction}
\lambda\cdot f(x)=\sum_{y\in{N(x)}}f(y)
\end{equation}
holds for any vertex $x\in V$. The set of functions $f:V\longrightarrow{\mathbb{R}}$ satisfying (\ref{Eq:Eigenfunction}) for any vertex $x\in V$ is called a {\em $\lambda$-eigenspace} of $G$. The {\em support} of a function $f:V\longrightarrow{\mathbb{R}}$ is the set $S(f)=\{x\in V~|~f(x)\neq 0\}$.
Denote $|f|=|S(f)|$.

Given a graph $G$, denote by $U(G)$ the set of all real-valued functions defined on the vertex set of $G$.
Note that the set $U(G)$ forms a vector space over $\mathbb{R}$.

The {\em $n$-dimensional hypercube} $H(n)$ is defined as follows.
The vertex set of $H(n)$ is $\mathbb{Z}_{2}^n$, and two vertices are adjacent if they differ in exactly one coordinate.
This graph has $n+1$ distinct eigenvalues $\lambda_i(n)=n-2i$, where $0\leq i\leq n$.
Denote by $U_{i}(n)$ the $\lambda_{i}(n)$-eigenspace of $H(n)$. The direct sum of subspaces
$$U_i(n)\oplus U_{i+1}(n)\oplus\ldots\oplus U_j(n)$$ for $0\leq i\leq j\leq n$ is denoted by $U_{[i,j]}(n)$. Denote $U(n)=U(H(n))$.

Let $G_1=(V_1,E_1)$ and $G_2=(V_2,E_2)$ be two graphs.
The {\em Cartesian product} $G_1\square G_2$ of graphs $G_1$ and $G_2$ is defined as follows.
The vertex set of $G_1\square G_2$ is $V_1\times V_2$;
and any two vertices $(x_1,y_1)$ and $(x_2,y_2)$ are adjacent if and only if either
$x_1=x_2$ and $y_1$ is adjacent to $y_2$ in $G_2$, or
$y_1=y_2$ and $x_1$ is adjacent to $x_2$ in $G_1$.

Suppose $G_1=(V_1,E_1)$ and $G_2=(V_2,E_2)$ are two graphs. Let $f_1:V_1\longrightarrow{\mathbb{R}}$ and $f_2:V_2\longrightarrow{\mathbb{R}}$.
Denote $G=G_1\square G_2$.
We define the {\em tensor product} $f_1\otimes f_2$  on the vertices of $G$ by the following rule:
$$(f_1\otimes f_2)(x,y)=f_1(x)f_2(y)$$ for $(x,y)\in V(G)=V_1\times V_2$.

Let $f$ be a real-valued function defined on the vertices of $H(n)$ and let $k\in \{0,1\}$, $r\in\{1,\ldots,n\}$.
We define a function $f_{k}^{r}$ on the vertices of $H(n-1)$ as follows:
for any vertex $y=(y_1,\ldots,y_{r-1},y_{r+1},\ldots,y_n)$ of $H(n-1)$
$$f_{k}^{r}(y)=f(y_1,\ldots,y_{r-1},k,y_{r+1},\ldots,y_n).$$

For a vector $u\in \mathbb{Z}_{2}^n$, where $u=(u_1,\ldots,u_n)$, we define a function $\chi_u$ on the vertices of $H(n)$ as follows:
$$\chi_u(x_1,\ldots,x_n)=(-1)^{u_1x_1+\ldots+u_nx_n}.$$
The functions $\chi_u$, where $u\in \mathbb{Z}_{2}^n$, are also known as the characters of the group $\mathbb{Z}_{2}^n$.

The {\em weight} of a vector $x\in \mathbb{Z}_{2}^n$, denoted by $\mathrm{wt}(x)$, is the number of its non-zero coordinates.

Let $A$ and $B$ be two finite subsets of $\mathbb{Z}$. Denote $$A+B=\{c\in \mathbb{Z}~|~c=a+b,a\in A,b\in B\}.$$

Let $\{i_1,\ldots,i_m\}$ be an $m$-element subset of $\{1,2,\ldots,n\}$ and let $a_i\in \{0,1\}$ for all $1\leq i\leq m$.
Denote $$\Gamma_{i_1,\ldots,i_m}^{a_1,\ldots,a_m}=\{(x_1,\ldots,x_n)\in \mathbb{Z}_{2}^n~|~x_{i_1}=a_1,\ldots,x_{i_m}=a_m\}.$$
For $m\in \{0,1,\ldots,n\}$, a set $\Gamma\subseteq \mathbb{Z}_{2}^n$ is called an {\em $(n-m)$-face} if there exist an $m$-element subset $\{i_1,\ldots,i_m\}$ of
$\{1,2,\ldots,n\}$ and numbers
$a_1,\ldots,a_m\in \{0,1\}$ such that $\Gamma=\Gamma_{i_1,\ldots,i_m}^{a_1,\ldots,a_m}$.


Recall that the set $U(n)$ forms a vector space over $\mathbb{R}$. We define an inner product on this vector space as follows:
$$\langle f,g \rangle=\frac{1}{2^n}\sum_{x\in \mathbb{Z}_{2}^n}f(x)g(x)$$
Two functions $f\in U(n)$ and $g\in U(n)$ are called {\em orthogonal} if $\langle f,g \rangle=0$.

A pair $\{T_0,T_1\}$ of two disjoint nonempty subsets of $\mathbb{Z}_2^{n}$ is called a {\em $[t]$-trade} in $H(n)$ if every $(n-t)$-face contains the same number of elements from $T_0$ and from $T_1$.
For a subset $A$ of $\mathbb{Z}_2^n$, let $\mathbf{1}_A$ denote the characteristic function of $A$ in $\mathbb{Z}_2^n$.

For every non-negative integer $r$ and every positive integer $n\geq r$, the {\em Reed–Muller code} $\mathcal{RM}(r,n)$ of order $r$ is
the set of all $n$-variable Boolean functions of algebraic degree at most $r$.

Let $G$ be a bipartite graph with parts $V_1$ and $V_2$. Suppose that $f$ is a real-valued function defined on the vertices of $G$.
We define a function $f'$ on the vertices of $G$ by the following rule:
$$
f'(x)=\begin{cases}
f(x),&\text{if $x\in V_1$;}\\
-f(x),&\text{if $x\in V_2$.}
\end{cases}
$$

For a function $f\in U(n)$, we define a function $\widetilde{f}$ on the vertices of $H(n)$ as follows:
$$\widetilde{f}(x_1,\ldots,x_n)=(-1)^{x_1+\ldots+x_n}\cdot f(x_1,\ldots,x_n).$$

Any function $f\in U(n)$ can be uniquely represented in the following form:
$$f=\sum_{i=0}^{n}f_i,$$
where $f_i\in U_i(n)$ for any $0\leq i\leq n$.
The {\em spectrum} of a function $f\in U(n)$ is the set $$\mathrm{Spec}(f)=\{0\leq i\leq n~|~f_i\not\equiv 0\}.$$

Two functions $f\in U(n)$ and $g\in U(n)$ are called {\em equivalent} if there exist an automorphism $\pi$ of $H(n)$ and a real non-zero constant $c$ such that the equality $g(x)=c\cdot f(\pi(x))$ holds for any vertex $x$ of $H(n)$. We denote this equivalence by $f\sim g$.


\section{Preliminaries}\label{Sec:Prelim}
In this section, we give preliminary results.
The following lemma is a special case of Corollary 1 proved in \cite{VV19}.
\begin{lemma}\label{Lemma:Product}
Let $f_1\in U_i(m)$ and $f_2\in U_j(n)$. Then $f_1\otimes f_2\in U_{i+j}(m+n)$.
\end{lemma}
The following result is a special case of Lemma 4 proved in \cite{VV19}.
\begin{lemma}\label{Lemma:Reduction}
Let $f\in{U_{[i,j]}(n)}$ and $r\in\{1,2,\ldots,n\}$. Then the following statements are true:
\begin{enumerate}
\item $f_{0}^{r}-f_{1}^{r}\in{U_{[i-1,j-1]}(n-1)}$.
\item $f_{0}^{r}+f_{1}^{r}\in{U_{[i,j]}(n-1)}$.
\item $f_{k}^{r}\in{U_{[i-1,j]}(n-1)}$ for $k\in \{0,1\}$.
\end{enumerate}
\end{lemma}

\begin{lemma}\label{Lemma:NiceRepr}
Let $f\in{U_{[i,j]}(n)}$ and $r\in\{1,2,\ldots,n\}$. Then there are functions $g$ and $h$ such that $f_{0}^{r}=g+h$, $f_{1}^{r}=g-h$ and $g\in U_{[i,j]}(n-1)$,
$h\in U_{[i-1,j-1]}(n-1)$.
\end{lemma}
\begin{proof}
Denote $g=\frac{1}{2}(f_{0}^{r}+f_{1}^{r})$ and $h=\frac{1}{2}(f_{0}^{r}-f_{1}^{r})$. Then we have $f_{0}^{r}=g+h$ and $f_{1}^{r}=g-h$.
In addition, by Lemma \ref{Lemma:Reduction} we obtain that $g\in U_{[i,j]}(n-1)$ and $h\in U_{[i-1,j-1]}(n-1)$.
\end{proof}
We will use Lemma \ref{Lemma:NiceRepr} in the proof of Lemma \ref{Lemma:3Values}.
The following two properties of the characters of $\mathbb{Z}_{2}^n$ are well-known.
\begin{lemma}\label{Lemma:Characters}
The following statements hold:
\begin{enumerate}

  \item The set $\{\chi_u~|~u\in \mathbb{Z}_{2}^n\}$ forms an orthonormal basis of the vector space $U(n)$.

  \item For every $0\leq i\leq n$, the set $\{\chi_u~|~u\in \mathbb{Z}_{2}^n, \mathrm{wt}(u)=i\}$ forms a basis of the vector space $U_i(n)$.

\end{enumerate}
\end{lemma}
We will use Lemma \ref{Lemma:Characters} for the proofs of Lemmas \ref{Lemma:BasisCharPr}, \ref{Lemma:SumSpectr}, \ref{Lemma:CorrImm} and \ref{Lemma:SpecFunctions}.
The following result about the Cartesian product of graphs is well-known.

\begin{lemma}\label{Lemma:OrtBasisPr}
Let $G_1$ and $G_2$ be graphs with $m$ and $n$ vertices.
If $f_1,\ldots,f_m$ and $g_1,\ldots,g_n$ are orthogonal bases for the vector spaces $U(G_1)$ and $U(G_2)$,
then the set $$\{f_i\otimes g_j~|~i\in \{1,\ldots,m\},j\in \{1,\ldots,n\}\}$$ forms an orthogonal basis of the vector space $U(G_1\square G_2)$.
\end{lemma}

Using Lemmas \ref{Lemma:Characters} and \ref{Lemma:OrtBasisPr}, we immediately obtain the following result.
\begin{lemma}\label{Lemma:BasisCharPr}
The set $\{\chi_u\otimes \chi_v~|~u\in \mathbb{Z}_{2}^m,v\in \mathbb{Z}_{2}^n\}$  forms an orthogonal basis of the vector space $U(m+n)$.
\end{lemma}

\begin{lemma}\label{Lemma:SumSpectr}
Let $f_1\in U(m)$ and $f_2\in U(n)$. Then $$\mathrm{Spec}(f_1\otimes f_2)=\mathrm{Spec}(f_1)+\mathrm{Spec}(f_2).$$
\end{lemma}
\begin{proof}
It follows from Lemmas \ref{Lemma:Characters}, \ref{Lemma:BasisCharPr} and \ref{Lemma:Product}.
\end{proof}
We will use Lemma \ref{Lemma:SumSpectr} in the proof of Lemma \ref{Lemma:OptimalConstruction}.
The following theorem is a combination of the results proved in \cite{V21} (see \cite[Theorems 3 and 4]{V21}).
\begin{theorem}\label{Th:Bound}
Let $f\in U_{[i,j]}(n)$ and $f\not\equiv 0$. Then the following statements hold:
\begin{enumerate}
  \item If $i+j\geq n$, then $|f|\geq 2^{i}$ and this bound is sharp.

  \item If $i+j\leq n$, then $|f|\geq 2^{n-j}$ and this bound is sharp.
\end{enumerate}

\end{theorem}
We will use Theorem \ref{Th:Bound} in the proof of Lemma \ref{Lemma:3Values}.

\begin{lemma}\label{Lemma:CorrImm}
 Let $f\in U_i(n)$, where $1\leq i\leq n$. Then for every $(n-i+1)$-face $\Gamma$ it holds $\sum_{x\in \Gamma}f(x)=0$.
\end{lemma}
\begin{proof}
Suppose that $\Gamma$ is an $(n-i+1)$-face.
It is easy to check that the function $\mathbf{1}_\Gamma$ is orthogonal to $\chi_u$ for any $u\in \mathbb{Z}_{2}^n$ of weight $i$.
Then by Lemma \ref{Lemma:Characters} we obtain that $\mathbf{1}_\Gamma$ is orthogonal to an arbitrary function from the space $U_i(n)$.
So, the functions $\mathbf{1}_\Gamma$ and $f$ are orthogonal and we have
$\sum_{x\in \Gamma}f(x)=0$.
\end{proof}

The following result was obtained in \cite{K18} (see the last paragraph in the proof of Theorem 1).
\begin{lemma}\label{Lemma:CharFunction}
Let $\{T_0,T_1\}$ be a $[t]$-trade in $H(n)$. Then $\mathbf{1}_{T_0\cup T_1}\in \mathcal{RM}(n-t-1,n)$.
\end{lemma}

The following fact is well known in coding theory (for example, see \cite[Chapter 13, Theorem 5]{MWS77} or \cite[Chapter 4, Theorem 8]{C20}).
\begin{lemma}\label{Lemma:CodewordsRM}
Any Boolean function from $\mathcal{RM}(r,n)$ of weight $2^{n-r}$ is the characteristic function of an $(n-r)$-dimensional affine subspace
of $\mathbb{Z}_2^n$.
\end{lemma}

The following lemma was proved in \cite{K18}.
\begin{lemma}[\cite{K18}, Proposition 1]\label{Lemma:SpaceSplit}
An affine subspace $T\subset \mathbb{Z}_2^n$ of dimension $t+1$ can be split into a $[t]$-trade $\{T_0,T_1\}$
if and only if it is a translation of the linear span of mutually disjoint base subsets.
\end{lemma}
We will use Lemmas \ref{Lemma:CorrImm}, \ref{Lemma:CharFunction}, \ref{Lemma:CodewordsRM} and \ref{Lemma:SpaceSplit} in the proof of Theorem \ref{Th:Main}.
The following fact is well known in spectral graph theory (for example, see \cite[Section 1.3.6]{BH12}).
\begin{lemma}\label{Lemma:Bipartite}
Let $G$ be a bipartite graph. If $f$ is a $\lambda$-eigenfunction of $G$, then $f'$ is a (-$\lambda$)-eigenfunction of $G$.
\end{lemma}

Since $H(n)$ is bipartite and $\lambda_{i}(n)=-\lambda_{n-i}(n)$, by Lemma \ref{Lemma:Bipartite} we immediately obtain the
following result.
\begin{lemma}\label{Lemma:SymmetricSpaces}
If $f\in U_i(n)$, then $\widetilde{f}\in U_{n-i}(n)$.
\end{lemma}

Using the previous lemma for $U_{k}(n)$, where $i\le k\le j$, we obtain the following result.
\begin{lemma}\label{Lemma:SymmetricIntSpaces}
If $f\in U_{[i,j]}(n)$, then $\widetilde{f}\in U_{[n-j,n-i]}(n)$.
\end{lemma}
We will use Lemma \ref{Lemma:SymmetricIntSpaces} in the proof of Theorem \ref{Th:Main}.

\section{Constructions of functions with the minimum cardinality of the support}\label{Sec:Constructions}
In this section, we give constructions of functions that are optimal in the space $U_{[i,j]}(n)$. We also find the spectrum of these functions.

For $k\geq 1$, we define a function $\varphi_k$ on the vertices of $H(k)$ by the following rule:
$$
\varphi_k(x)=\begin{cases}
1,&\text{if $x$ is the all-zeros vector;}\\
-1,&\text{if $x$ is the all-ones vector;}\\
0,&\text{otherwise.}
\end{cases}
$$

For $k\geq 1$, we define a function $\psi_k$ on the vertices of $H(k)$ by the following rule:
$$
\psi_k(x)=\begin{cases}
1,&\text{if $x$ is the all-zeros vector;}\\
1,&\text{if $x$ is the all-ones vector;}\\
0,&\text{otherwise.}
\end{cases}
$$

For $k\geq 1$, we define a function $I_k$ on the vertices of $H(k)$ by the following rule:
$$
I_k(x)=\begin{cases}
1,&\text{if $x$ is the all-zeros vector;}\\
0,&\text{otherwise.}
\end{cases}
$$

The functions $\varphi_3$, $\psi_3$ and $I_3$ are shown in Figure \ref{F}:

\begin{figure}[H]
\includegraphics[scale=0.27]{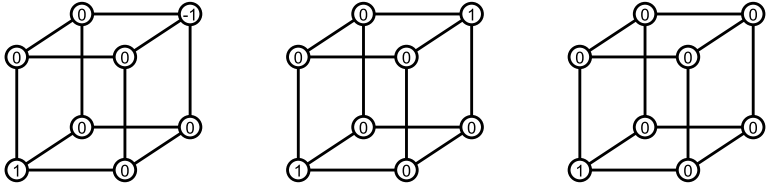}
\caption{Functions $\varphi_3$, $\psi_3$ and $I_3$ in $H(3)$}\label{F}
\end{figure}

\begin{lemma}\label{Lemma:SpecFunctions}
The following statements are true:
\begin{enumerate}
\item $\mathrm{Spec}(\varphi_{2k+1})=\{1,3,\ldots,2k+1\}$ for $k\geq 1$ and $\mathrm{Spec}(\varphi_1)=\{1\}$.

\item  $\mathrm{Spec}(\varphi_{2k})=\{1,3,\ldots,2k-1\}$ for $k\geq 2$ and $\mathrm{Spec}(\varphi_2)=\{1\}$.

\item $\mathrm{Spec}(\psi_{2k+1})=\{0,2,\ldots,2k\}$ for $k\geq 1$ and $\mathrm{Spec}(\psi_1)=\{0\}$.

\item $\mathrm{Spec}(I_k)=\{0,1,\ldots,k\}$ for $k\geq 1$.
\end{enumerate}
\end{lemma}
\begin{proof}
Let us consider the function $\varphi_n$.
By Lemma \ref{Lemma:Characters}, there exist the real numbers $c_u$, where $u\in \mathbb{Z}_{2}^n$, such that
$$\varphi_n=\sum_{u\in \mathbb{Z}_{2}^n}c_u\chi_u.$$
Then we have $$\langle \varphi_n,\chi_u \rangle=\langle \sum_{u\in \mathbb{Z}_{2}^n}c_u\chi_u,\chi_u \rangle=c_u\langle \chi_u,\chi_u \rangle=c_u.$$
On the other hand, $$\langle \varphi_n,\chi_u \rangle=\frac{1}{2^n}\sum_{x\in \mathbb{Z}_{2}^n}\varphi_n(x)\chi_u(x)=\frac{1}{2^n}(1-(-1)^{u_1+\ldots+u_n}).$$
Hence $$c_u=\frac{1}{2^n}(1-(-1)^{u_1+\ldots+u_n})$$ for any $u\in \mathbb{Z}_{2}^n$. So, we have
$$
c_u=\begin{cases}
\frac{1}{2^{n-1}},&\text{if $\mathrm{wt}(u)$ is odd;}\\
0,&\text{if $\mathrm{wt}(u)$ is even.}
\end{cases}
$$
Using Lemma \ref{Lemma:Characters}, we obtain that $\mathrm{Spec}(\varphi_n)$ consists of odd numbers belonging to the set $\{1,\ldots,n\}$.

The proofs for the functions $\psi_n$ and $I_n$ are similar.
\end{proof}

\begin{lemma}\label{Lemma:OptimalConstruction}
Let $n$ be a positive integer and $n=n_1+\ldots+n_k+m_1+\ldots+m_{\ell}+r$, where $n_1,\ldots,n_k$ are odd positive integers, $m_1,\ldots,m_{\ell}$ are even positive integers, $k$, $\ell$ and $r$ are nonnegative integers.
Then the following statements hold:

\begin{enumerate}
    \item Let $f=\varphi_{n_1}\otimes\cdots \otimes \varphi_{n_k}\otimes \varphi_{m_1}\otimes\cdots \otimes \varphi_{m_{\ell}}\otimes I_r$.
Then $f\in U_{[k+\ell,n-\ell]}(n)$ and $|f|=2^{k+\ell}$. Moreover, $$\mathrm{Spec}(f)=\{k+\ell,k+\ell+1,\ldots,n-\ell\}$$ for $r>0$ and $$\mathrm{Spec}(f)=\{k+\ell,k+\ell+2,\ldots,n-\ell\}$$ for $r=0$.

 \item Let $f=\psi_{n_1}\otimes\cdots \otimes \psi_{n_k}\otimes \varphi_{m_1}\otimes\cdots \otimes \varphi_{m_{\ell}}\otimes I_r$.
Then $f\in U_{[\ell,n-k-\ell]}(n)$ and $|f|=2^{k+\ell}$. Moreover, $$\mathrm{Spec}(f)=\{\ell,\ell+1,\ldots,n-k-\ell\}$$ for $r>0$ and $$\mathrm{Spec}(f)=\{\ell,\ell+2,\ldots,n-k-\ell\}$$ for $r=0$.

\end{enumerate}
\end{lemma}
\begin{proof}
Let us consider the first case.
By Lemma \ref{Lemma:SumSpectr} we have
$$\mathrm{Spec}(f)=\mathrm{Spec}(\varphi_{n_1})+\ldots+\mathrm{Spec}(\varphi_{n_k})+\ldots+\mathrm{Spec}(\varphi_{m_1})+\ldots+\mathrm{Spec}(\varphi_{m_{\ell}})+\mathrm{Spec}(I_r).$$
Then applying Lemma \ref{Lemma:SpecFunctions}, we obtain that $$\mathrm{Spec}(f)=\{k+\ell,k+\ell+1,\ldots,n-\ell\}$$ for $r>0$ and $$\mathrm{Spec}(f)=\{k+\ell,k+\ell+2,\ldots,n-\ell\}$$ for $r=0$. Using the equality $|f_1\otimes f_2|=|f_1|\cdot |f_2|$, we see that $|f|=2^{k+\ell}$.

The proof for the second case is similar.
\end{proof}

\section{Main results}\label{Sec:Main}
In this section, we prove the main theorem of this paper. Firstly, we prove the following result.

\begin{lemma}\label{Lemma:3Values}
Let $f\in U_{[i,j]}(n)$, where $i+j\geq n$. If $|f|=2^i$, then $f$ takes values from the set $\{-a,0,a\}$, where $a$ is a positive real number.
\end{lemma}
\begin{proof}
Let us prove this lemma by induction on $n$, $i$ and $j$. If $i=0$, then $|f|=1$ and the claim of the lemma holds. So, we can assume that $i\geq 1$. If $n=1$, then $i=j=1$. Then $f\in U_1(1)$ and the claim of the lemma
holds.

Let us prove the induction step for $n\geq 2$ and $i\geq 1$. Let us consider the functions $f_{0}^{n}$ and $f_{1}^{n}$.
Denote $f_k=f_{k}^{n}$ for $k\in \{0,1\}$.
Lemma \ref{Lemma:NiceRepr} implies that there are functions $g$ and $h$ such that $f_0=g+h$, $f_1=g-h$ and $g\in U_{[i,j]}(n-1)$,
$h\in U_{[i-1,j-1]}(n-1)$.
Let us consider two cases.

In the first case we suppose that $g\equiv 0$. In this case we have $|h|=\frac{1}{2}|f|=2^{i-1}$.
Let us show that $i+j\geq n+1$. Indeed, if $i+j=n$, then $|h|\geq 2^i$ due to Theorem \ref{Th:Bound}.
Since $|h|=2^{i-1}$, we get a contradiction. So, in this case we have $i+j\geq n+1$.
Applying the induction assumption for $h$, we obtain that $h$ takes values from the set $\{-a,0,a\}$, where $a$ is a positive real number.
Therefore, $f$ also takes values from the set $\{-a,0,a\}$.

In the second case we suppose that $g\not\equiv 0$. Since $g\in U_{[i,j]}(n-1)$, by Theorem \ref{Th:Bound} we obtain that $|g|\geq 2^i$.
Then we have $$|f|=|f_0|+|f_1|\geq |f_0+f_1|=|g|\geq 2^i.$$ Therefore $|g|=|f|=2^i$.
Applying the induction assumption for $g$, we obtain that $g$ takes values from the set $\{-a',0,a'\}$, where $a'$ is a positive real number.
Since $|f|=|g|$, we have $h(x)\in \{-g(x),g(x)\}$ for every vertex $x$ of $H(n-1)$.
Thus, $f$ takes values from the set $\{-2a',0,2a'\}$.
\end{proof}

The main result of this paper is the following.

\begin{theorem}\label{Th:Main}
The following statements hold:

\begin{enumerate}
  \item Let $f\in U_{[i,j]}(n)$, where $i+j\geq n$. The equality $|f|=2^i$ holds if and only if $f$ is equivalent to
  $$\varphi_{n_1}\otimes\cdots \otimes \varphi_{n_k}\otimes \varphi_{m_1}\otimes\cdots \otimes \varphi_{m_{\ell}}\otimes I_r,$$ where
 $n=n_1+\ldots+n_k+m_1+\ldots+m_{\ell}+r$, $n_1,\ldots,n_k$ are odd positive integers, $m_1,\ldots,m_{\ell}$ are even positive integers, $k$, $\ell$ and $r$ are nonnegative integers, $k+\ell=i$ and $\ell\geq n-j$.

  \item Let $f\in U_{[i,j]}(n)$, where $i+j\leq n$. The equality $|f|=2^{n-j}$ holds if and only if $f$ is equivalent to
  $$\psi_{n_1}\otimes\cdots \otimes \psi_{n_k}\otimes \varphi_{m_1}\otimes\cdots \otimes \varphi_{m_{\ell}}\otimes I_r,$$ where
 $n=n_1+\ldots+n_k+m_1+\ldots+m_{\ell}+r$, $n_1,\ldots,n_k$ are odd positive integers, $m_1,\ldots,m_{\ell}$ are even positive integers, $k$, $\ell$ and $r$ are nonnegative integers, $k+\ell=n-j$ and $\ell\geq i$.
\end{enumerate}

\end{theorem}

\begin{proof}

1. Suppose that $|f|=2^i$.
If $i=0$, then $j=n$. In this case $|f|=1$. Therefore, $f\sim I_n$ and the claim of the theorem holds.
In what follows in the proof of Theorem \ref{Th:Main} for $i+j\geq n$ we can assume that $i\geq 1$.

Let us consider a pair $\{T_0,T_1\}$, where $T_0=\{x\in \mathbb{Z}_2^n~|~f(x)>0\}$ and  $T_1=\{x\in \mathbb{Z}_2^n~|~f(x)<0\}$.
Lemmas \ref{Lemma:CorrImm} and \ref{Lemma:3Values} imply that every $(n-i+1)$-face contains the same number of elements from $T_0$ and from $T_1$.
So, $\{T_0,T_1\}$ is an $[i-1]$-trade in $H(n)$. Lemma \ref{Lemma:CharFunction} implies that $$\mathbf{1}_{T_0\cup T_1}\in \mathcal{RM}(n-i,n).$$
Since $|\mathbf{1}_{T_0\cup T_1}|=|f|$, we have $|\mathbf{1}_{T_0\cup T_1}|=2^i$.
Then by Lemma \ref{Lemma:CodewordsRM} we have that $\mathbf{1}_{T_0\cup T_1}$ is the characteristic function of an $i$-dimensional affine subspace
of $\mathbb{Z}_2^n$. Applying Lemma \ref{Lemma:SpaceSplit}, we obtain that $$f\sim \varphi_{t_1}\otimes\cdots \otimes \varphi_{t_i}\otimes I_r,$$
where $n=t_1+\ldots+t_i+r$, $t_1,\ldots,t_i$ are positive integers and $r$ is a nonnegative integer.
Suppose that the set $\{t_1,\ldots,t_i\}$ consists of $k$ odd numbers $n_1,\ldots,n_k$ and $\ell$ even numbers $m_1,\ldots,m_{\ell}$.
Then $$f\sim \varphi_{n_1}\otimes\cdots \otimes \varphi_{n_k}\otimes \varphi_{m_1}\otimes\cdots \otimes \varphi_{m_{\ell}}\otimes I_r.$$
Using Lemma \ref{Lemma:OptimalConstruction}, we see that $f\in U_{[i,n-\ell]}(n)$ and $n-\ell\in \mathrm{Spec}(f)$.
Since $f\in U_{[i,j]}(n)$, we obtain $\ell\geq n-j$.

Conversely, suppose that $$f\sim \varphi_{n_1}\otimes\cdots \otimes \varphi_{n_k}\otimes \varphi_{m_1}\otimes\cdots \otimes \varphi_{m_{\ell}}\otimes I_r,$$ where
$n=n_1+\ldots+n_k+m_1+\ldots+m_{\ell}+r$, $n_1,\ldots,n_k$ are odd positive integers, $m_1,\ldots,m_{\ell}$ are even positive integers, $k$, $\ell$ and $r$ are nonnegative integers, $k+\ell=i$ and $\ell\geq n-j$. Lemma \ref{Lemma:OptimalConstruction} implies that $f\in U_{[k+\ell,n-\ell]}(n)$ and $|f|=2^{k+\ell}$.
Since $k+\ell=i$ and $\ell\geq n-j$, we have $f\in U_{[i,j]}(n)$ and $|f|=2^i$.

2. Suppose that $|f|=2^{n-j}$.
Lemma \ref{Lemma:SymmetricIntSpaces} implies that $\widetilde{f}\in U_{[n-j,n-i]}(n)$. Note that $|\widetilde{f}|=|f|=2^{n-j}$.
By the first case of this theorem we obtain that $\widetilde{f}\sim v$, where $$v=\varphi_{n_1}\otimes\cdots \otimes \varphi_{n_k}\otimes \varphi_{m_1}\otimes\cdots \otimes \varphi_{m_{\ell}}\otimes I_r,$$
$n=n_1+\ldots+n_k+m_1+\ldots+m_{\ell}+r$, $n_1,\ldots,n_k$ are odd positive integers, $m_1,\ldots,m_{\ell}$ are even positive integers, $k$, $\ell$ and $r$ are nonnegative integers, $k+\ell=n-j$ and $\ell\geq i$.
Using the equality $\widetilde{f_1\otimes f_2}=\widetilde{f_1}\otimes \widetilde{f_2}$, we obtain that $$\widetilde{v}=\psi_{n_1}\otimes\cdots \otimes \psi_{n_k}\otimes \varphi_{m_1}\otimes\cdots \otimes \varphi_{m_{\ell}}\otimes I_r.$$
Therefore, we have $$f\sim \psi_{n_1}\otimes\cdots \otimes \psi_{n_k}\otimes \varphi_{m_1}\otimes\cdots \otimes \varphi_{m_{\ell}}\otimes I_r.$$

Conversely, suppose that $$f\sim \psi_{n_1}\otimes\cdots \otimes \psi_{n_k}\otimes \varphi_{m_1}\otimes\cdots \otimes \varphi_{m_{\ell}}\otimes I_r,$$ where
$n=n_1+\ldots+n_k+m_1+\ldots+m_{\ell}+r$, $n_1,\ldots,n_k$ are odd positive integers, $m_1,\ldots,m_{\ell}$ are even positive integers, $k$, $\ell$ and $r$ are nonnegative integers, $k+\ell=n-j$ and $\ell\geq i$. Lemma \ref{Lemma:OptimalConstruction} implies that $f\in U_{[\ell,n-k-\ell]}(n)$ and $|f|=2^{k+\ell}$.
Since $k+\ell=n-j$ and $\ell\geq i$, we have $f\in U_{[i,j]}(n)$ and $|f|=2^{n-j}$.
\end{proof}

Applying Theorem \ref{Th:Main} for $i=j$, we obtain the following result.
\begin{corollary}\label{Cor:}
The following statements hold:

\begin{enumerate}
  \item Let $f\in U_i(n)$, where $i\geq \frac{n}{2}$. The equality $|f|=2^i$ holds if and only if $f$ is equivalent to $\varphi_1^{2i-n}\otimes \varphi_2^{n-i}$.

  \item Let $f\in U_i(n)$, where $i\leq \frac{n}{2}$. The equality $|f|=2^{n-i}$ holds if and only if $f$ is equivalent to $\psi_1^{n-2i}\otimes \varphi_2^{i}$.
\end{enumerate}

\end{corollary}

Finally, we illustrate Theorem \ref{Th:Main} in the following examples:

\begin{example}
Let $n=4$, $i=2$ and $j=3$. There are exactly two partitions of $4$ such that $k+\ell=2$ and $\ell\geq 1$: $4=2+2$ and $4=1+2+1$.
These partitions correspond to the functions $\varphi_2\otimes \varphi_2$ and $\varphi_1\otimes \varphi_2\otimes I_1$ respectively.
\end{example}

\begin{example}
Let $n=3$, $i=0$ and $j=2$. There are exactly three partitions of $3$ such that $k+\ell=1$ and $\ell\geq 0$: $3=2+1$, $3=3$ and $3=1+2$.
These partitions correspond to the functions $\varphi_2\otimes I_1$, $\psi_3$ and $\psi_1\otimes I_2$ respectively.
\end{example}

\section{Spectrum of optimal functions}\label{Sec:Spectr}
In this section, we discuss the spectrum of functions that are optimal in the space $U_{[i,j]}(n)$. Theorem \ref{Th:Main} and Lemma \ref{Lemma:OptimalConstruction} imply that the spectrum of such functions forms an arithmetic progression with common difference $1$ or $2$. More precisely, we have the following result.
\begin{corollary}\label{Cor:2}
The following statements hold:
\begin{enumerate}

  \item Let $f\in U_{[i,j]}(n)$, where $i+j\geq n$. If $|f|=2^i$, then $$\mathrm{Spec}(f)=\{i,i+d,\ldots,i+kd\},$$
  where $d\in \{1,2\}$, $k$ is non-negative integer and $i+kd\leq j$.

  \item Let $f\in U_{[i,j]}(n)$, where $i+j\leq n$. If $|f|=2^{n-j}$, then $$\mathrm{Spec}(f)=\{j-kd,j-(k-1)d,\ldots,j\},$$
  where $d\in \{1,2\}$, $k$ is non-negative integer and $j-kd\geq i$.
\end{enumerate}

\end{corollary}

Corollary \ref{Cor:2} implies that if $f\in U_{[i,j]}(n)$ and $\mathrm{Spec}(f)$ is not an arithmetic progression of a special kind,
then $|f|>\max(2^i,2^{n-j})$. For example, if $f\in U(n)$ and $\mathrm{Spec}(f)=\{0,3\}$, where $n\geq 3$, then $|f|>2^{n-3}$.
In view of these observations, it seems natural to consider the following question.

\begin{problem}
Let $n\geq 3$. Find $$\min_{f\in U(n), \mathrm{Spec}(f)=\{0,3\}}|f|.$$
\end{problem}

\end{document}